\documentclass{amsart}

\usepackage[utf8]{inputenc} 
\usepackage[T1]{fontenc} 
\usepackage[english]{babel}

\theoremstyle{plain} 
\newtheorem{theorem}{Theorem} 
\newtheorem{lemma}[theorem]{Lemma}

\theoremstyle{remark} 
\newtheorem{conjecture}{Conjecture}

\DeclareMathOperator{\mre}{Re}

\begin{document} 
\title{Coefficient estimates for $H^p$ spaces with $0<p<1$} 
\date{\today} 

\author{Ole Fredrik Brevig} 
\address{Department of Mathematical Sciences, Norwegian University of Science and Technology (NTNU), NO-7491 Trondheim, Norway} 
\email{ole.brevig@math.ntnu.no}

\author{Eero Saksman} 
\address{Department of Mathematical Sciences, Norwegian University of Science and Technology (NTNU), NO-7491 Trondheim, Norway}
\address{Department of Mathematics and Statistics, University of Helsinki, FI-00170 Helsinki, Finland}
\email{eero.saksman@helsinki.fi} 
\begin{abstract}
	Let $C(k,p)$ denote the smallest real number such that the estimate $|a_k|\leq C(k,p)\|f\|_{H^p}$ holds for every $f(z)=\sum_{n\geq0}a_n z^n$ in the $H^p$ space of the unit disc. We compute $C(2,p)$ for $0<p<1$ and $C(3,2/3)$, and identify the functions attaining equality in the estimate.
\end{abstract}


\thanks{Some of the work on the present paper was carried out during the workshop ``Operator related Function Theory'' at the Erwin Schr\"odinger Institute. The authors gratefully acknowledge the support of the ESI}

\maketitle

\section{Introduction} For $0<p<\infty$, the Hardy space $H^p$ is comprised of the analytic functions $f$ in the unit disc $\mathbb{D} = \{z\in\mathbb{C}\,:\,|z|<1\}$ which satisfy
\[\|f\|_{H^p}^p = \lim_{r \to 1^-} \int_0^{2\pi} |f(re^{i\theta})|^p \, \frac{d\theta}{2\pi} < \infty.\]
The Hardy space $H^p$ is a Banach space when $1\leq p < \infty$ and a quasi-Banach space when $0<p<1$. For an integer $k\geq1$, let $C(k,p)$ denote the smallest real number such that
\[|a_k| \leq C(k,p)\|f\|_{H^p}\]
holds for every $f(z) = \sum_{n\geq0} a_n z^n$ in $H^p$. In other words, $C(k,p)$ is the norm of the bounded linear functional $L_k(f) = a_k$ on $H^p$.

In the range $1\leq p < \infty$ it follows readily from the triangle inequality and H\"older's inequality that $C(k,p)=1$ for every $k\geq1$. Estimates for $C(k,p)$ when $0<p<1$ were first obtained by Hardy and Littlewood \cite{HL32}, who proved that there is a constant $C_p\geq1$ such that $C(k,p) \leq C_p k^{1/p-1}$ holds for every $k\geq1$. 

In this paper we are interested in computing $C(k,p)$ explicitly in the non-trivial range $0<p<1$. For this purpose it is fruitful to express this quantity via the associated linear extremal problem 
\begin{equation}\label{eq:extremal} 
	C(k,p) = \sup\left\{\mre{\frac{f^{(k)}(0)}{k!}}\,:\, \|f\|_{H^p}=1\right\}. 
\end{equation}
A normal family argument implies that there are functions $f$ in the unit ball of $H^p$ attaining the supremum \eqref{eq:extremal}. In a recent joint paper with Bondarenko and Seip \cite{BBSS}, we proved that the extremal function for $k=1$ in \eqref{eq:extremal} is given by 
\begin{equation}\label{eq:C1pextremal} 
	f(z) = \left(1-\frac{p}{2}\right)^{\frac{1}{p}}\left(1 + \sqrt{\frac{p}{2-p}} z\right)^{\frac{2}{p}}, 
\end{equation}
up to rotations $f(z) \mapsto e^{-i\theta}f(e^{i\theta}z)$. Consequently, we found that 
\begin{equation}\label{eq:C1p} 
	C(1,p) = \sqrt{\frac{2}{p}}\left(1-\frac{p}{2}\right)^{\frac{1}{p}-\frac{1}{2}}. 
\end{equation}
The approach used in \cite{BBSS} is to write $f$ in the unit ball of $H^p$ as $f=g h^{2/p-1}$, where $g$ and $h$ are in the unit ball of $H^2$ and $h$ does not vanish in $\mathbb{D}$. If the coefficient sequences of $g$ and $h^{2/p-1}$ are $(b_n)_{n\geq0}$ and $(c_n)_{n\geq0}$, respectively, then 
\begin{equation}\label{eq:cauchyproduct} 
	\frac{f^{k}(0)}{k!} = \sum_{j=0}^k b_j c_{k-j}. 
\end{equation}
For any fixed non-vanishing $h$ in the unit ball of $H^2$, it is now easy to find the optimal $g$ in the unit ball of $H^2$ to maximize \eqref{eq:cauchyproduct} by the Cauchy--Schwarz inequality. This translates the linear extremal problem \eqref{eq:extremal} in $H^p$ to a non-linear extremal problem for non-vanishing functions in $H^2$.

By using the Cauchy--Schwarz inequality in this way and treating $g$ and $h$ as completely independent, we actually double the degree of the non-linear extremal problem. When $k=1$ this does not make the problem much harder, but already for $k=2$ this approach becomes computationally untenable.

For a class of linear extremal problems including \eqref{eq:extremal} on $H^p$ with $1\leq p < \infty$, there is a well-developed theory which yields that the extremal functions have a very specific structure (see e.g.~\cite[Sec.~8.4]{Duren}). The proof of this structure result relies on the fact that $H^p$ is a Banach space and duality arguments. These techniques do not apply for $0<p<1$, but we can replace them with a variational argument which goes back to F.~Riesz \cite{FRiesz19} and obtain the same result also for $0<p<1$. 

This structure result is a special case of a more general result on the structure of the solutions to the Carath\'{e}odory--Fej\'{e}r problem, which was extended from the range $1 \leq p < \infty$ to the range $0<p<1$ by Kabaila \cite{Kabaila60} (see also \cite[pp.~82--83]{Khavinson86} --- the latter reference actually develops a general theory that covers many related variational problems on $H^p$ spaces).  This extension to $0<p<1$ explicitly uses the structure of the solutions for $1\leq p<\infty$, while the variational argument presented in the present paper actually applies in the range $0<p<2$ without modification.

The information regarding the structure of the extremals $f$ for the linear extremal problem \eqref{eq:extremal} thus obtained shows that $g$ and $h$ in the factorization $f=g h^{2/p-1}$ are closely related. This greatly simplifies the non-linear extremal problem we have to solve in order to identify the extremals. Consequently, we are able to completely settle the case $k=2$.
\begin{theorem}\label{thm:C2p} 
	For $0<p<1$ we have
	\[C(2,p) = \frac{2}{p}\left(1-\frac{p}{2}\right)^{\frac{2}{p}-1}\]
	and, up to the rotations $f(z)\mapsto e^{-2i\theta}f(e^{i\theta}z)$, the extremal function in \eqref{eq:extremal} is
	\[f(z) = \left(1-\frac{p}{2}\right)^{\frac{2}{p}}\left(1 + \sqrt{\frac{2p}{2-p}}z + \frac{p}{2-p} z^2\right)^{\frac{2}{p}}.\]
\end{theorem}

Comparing \eqref{eq:C1p} and Theorem~\ref{thm:C2p}, we see the curious identity $C(2,p)=C(1,p)^2$. The next result demonstrates that the same relationship does not hold in general. 
\begin{theorem}\label{thm:C323} 
	We have
	\[C(3,2/3) = \sqrt{\frac{2\left(1103+33\sqrt{33}\,\right)}{1153}} = 1.4973\ldots\]
	and, up to the rotations $f(z)\mapsto e^{-3i\theta}f(e^{i\theta}z)$, the extremal function in \eqref{eq:extremal} is
	\[f(z) = \left(\frac{483-19\sqrt{33}}{1153}\right)^{\frac{3}{2}}\left(1+\frac{\sqrt{3+\frac{1}{3}\sqrt{33}}}{2}z+\frac{1+\sqrt{33}}{8}z^2 + \frac{\sqrt{15-\sqrt{33}}}{8}z^3\right)^3.\]
\end{theorem}

This paper is organized into four additional sections. In Section~\ref{sec:prelim} we recall some preliminaries about Hardy spaces and obtain the above-mentioned structure result for $0<p<1$. The proofs of Theorems~\ref{thm:C2p} and \ref{thm:C323} are presented, respectively, in Sections~\ref{sec:C2p} and \ref{sec:C323}. Section~\ref{sec:remarks} contains some concluding remarks, conjectures and discussions of related work.

\section{Preliminaries} \label{sec:prelim} 
In the present section, we will use several basic facts pertaining to Hardy spaces. We refer generally to the monograph \cite{Duren}, which contains most of what which we require. Our goal is to describe the structure of the extremals for bounded linear functionals $L_k$ on $H^p$, when $L_k(f)$ depends only on the first $k+1$ coefficients of the function $f(z)=\sum_{n\geq0} a_n z^n$. In the case $1\leq p < \infty$, this description is a consequence of a general theory of linear extremal problems for $H^p$ spaces developed by Macintyre, Rogosinski, Shapiro and Havinson (see e.g. \cite{Havinson51,MR50} and \cite[Ch.~8]{Duren}).

To set the stage for a discussion of their approach and ours, we recall that every $f$ in $H^p$ has non-tangential boundary limits
\[f(e^{i\theta}) = \lim_{r\to 1^-} f(re^{i\theta})\]
for almost every $e^{i\theta} \in \mathbb{T}=\{z\in\mathbb{C}\,:\, |z|=1\}$. It also holds that $\|f\|_{H^p}=\|f\|_{L^p(\mathbb{T})}$, so $H^p$ is identified with a subspace of $L^p(\mathbb{T})$, the latter defined in terms of the normalized Lebesgue arc length measure on $\mathbb{T}$. 

Every bounded linear functional $L$ on $H^p$, for $1\leq p<\infty$, can be represented in the inner product of $L^2(\mathbb{T})$ as
\[L(f) = \langle f, \varphi \rangle\]
for some analytic function $\varphi$ in $\mathbb{D}$ which is (at least) integrable on $\mathbb{T}$. Since $H^2$ is a Hilbert space, the analytic function $\varphi$ generating the functional is (up to a constant) equal to the extremal $f$ for the functional $L$. This fact leads naturally to the following.

Since $H^p$ is a Banach space when $1\leq p < \infty$, the Hahn--Banach theorem extends every bounded linear functional on $H^p$ to a bounded linear functional on $L^p(\mathbb{T})$ with the same norm. This makes it possible to formulate the dual extremal problem, which is to find an element $\psi$ of minimal norm in $L^{p^\ast}(\mathbb{T})$, where $1/p+1/p^\ast=1$, such that $L(f)=\langle f, \psi \rangle$. These two problems are closely related, and this can be exploited obtain a description of the structure of the extremals (and the structure of the element $\psi$ of minimal norm generating the functional) when the functional depends only on the first $k+1$ coefficients of $f$.

These techniques are not available to us in the range $0<p<1$, both since we cannot use the Hahn--Banach theorem and even if we could, $L^p(\mathbb{T})$ supports no non-trivial bounded linear functionals. We will therefore replace the duality approach outlined above with a variational argument essentially due to F.~Riesz \cite{FRiesz19}. See also \cite[Sec.~2]{Suffridge90} for a similar argument in a somewhat different context. Note that this method actually applies in the range $0<p<2$ without modification. We require two additional preliminary facts before proceeding.

Every function $f$ in $H^p$ can be written as $F = I O$, where $I$ is an inner function and $O$ is an outer function. In particular, $O$ does not vanish in $\mathbb{D}$ and $|I(e^{i\theta})|=1$ for almost every $e^{i\theta} \in \mathbb{T}$. This allows us to factor
\begin{equation}\label{eq:innerouter} 
	f = g h^{2/p-1} 
\end{equation}
where $g = I O^{p/2}$ and $h = O^{p/2}$. We note that $|g(e^{i\theta})|=|h(e^{i\theta})|=|f(e^{i\theta})|^{p/2}$ holds for almost every $e^{i\theta}\in\mathbb{T}$, which yields the norm equalities $\|f\|_{H^p}^p = \|g\|_{H^2}^2 = \|h\|_{H^2}^2$.

Let $H^\infty$ denote the algebra of all bounded analytic functions in $\mathbb{D}$, setting
\[\|\varphi\|_{H^\infty} = \sup_{z\in\mathbb{D}} |\varphi(z)|.\]
Recall that $H^\infty$ is the multiplier algebra of $H^p$, for $0<p<\infty$, i.e. the algebra of functions $\varphi$ such that $\varphi f$ is in $H^p$ for every $f$ in $H^p$. 

Here is the key variational lemma which will give the structure of the extremals as discussed above. We will only use the special case where $\varphi$ is a monomial, but the proof of the lemma in this special case is identical to the proof for the general case. 
\begin{lemma}\label{lem:variational} 
	Fix $0<p<2$. Suppose that $L$ is a bounded linear functional on $H^p$ and that $f$ is an extremal for $\mre{L(f)}$ with $\|f\|_{H^p}=1$. If $f = g h^{2/p-1}$ such that $\|g\|_{H^2}=\|h\|_{H^2}=1$ and $h$ does not vanish in $\mathbb{D}$, then it holds that
	\[L(\varphi f) = L(f)\langle \varphi, |h|^2 \rangle\]
	for every $\varphi \in H^\infty$. 
\end{lemma}
\begin{proof}
	Set $q=2/p-1>0$. By \eqref{eq:innerouter} the extremal $f$ in the unit ball of $H^p$ may be written as $gh^q$ where $g$ and $h$ are in the unit ball of $H^2$ and $h$ does not vanish in $\mathbb{D}$. If $\|\varphi\|_{H^\infty}=0$ there is nothing to prove, so we therefore assume that $\|\varphi\|_{H^\infty}>0$ and consider $0\leq\varepsilon<\|\varphi\|_{H^\infty}^{-1}$. A computation reveals that
	\[\|(1+\varepsilon \varphi) h\|_{H^2}^2 = 1 + 2\varepsilon \mre{\langle \varphi, |h|^2 \rangle} + \varepsilon^2 \|\varphi h\|_{H^2}^2,\]
	since $\|h\|_{H^2}=1$. Hence
	\[h_\varepsilon(z) = (1+\varepsilon \varphi(z)) h(z) \big(1 + 2\varepsilon \mre{\langle \varphi, |h|^2 \rangle} + \varepsilon^2 \|\varphi h\|_{H^2}^2\big)^{-\frac{1}{2}}\]
	satisfies $\|h_\varepsilon\|_{H^2}=1$. We then compute
	\[\frac{d}{d\varepsilon} h_\varepsilon(z) \Bigg|_{\varepsilon=0} = \varphi(z) h(z) - \frac{1}{2} h(z) 2\mre{\langle \varphi, |h|^2 \rangle}= h(z)\big(\varphi(z) - \mre{\langle \varphi, |h|^2 \rangle}\big).\]
	If $0 \leq \varepsilon < \|\varphi\|_{H^\infty}^{-1}$, then $h_\varepsilon^q$ is analytic in $\mathbb{D}$ owing to the fact that $1+\varepsilon \varphi$ and $h$ do not vanish in $\mathbb{D}$. Hence, by H\"older's inequality and the fact that $q>0$ we find that $f_\varepsilon = g h_\varepsilon^q$ is in the unit ball of $H^p$. Since $f$ is extremal for $\mre L$, clearly $\mre L(f) \geq \mre L(f_\varepsilon)$ for every $0\leq \varepsilon < \|\varphi\|_\infty^{-1}$. Using that the functional $L$ is bounded, we conclude that 
	\begin{multline*}
		0\geq \mre L \left(\frac{d}{d\varepsilon}f_\varepsilon \Bigg|_{\varepsilon=0}\right) = q \mre \big(L(\varphi f) - L(f) \mre{\langle \varphi, |h|^2 \rangle}\big) \\
		= q \mre \big(L(\varphi f) - L(f)\langle \varphi, |h|^2 \rangle \big). 
	\end{multline*}
	This inequality also holds when $\varphi$ is replaced by $-\varphi$ and $\pm i \varphi$, which implies that $L(\varphi f ) = L(f) \langle \varphi, |h|^2 \rangle$.
\end{proof}

One final preliminary result is required. The Fej\'er--Riesz theorem (see~\cite{Fejer16}) states that the trigonometric polynomial $Q(\theta)=\sum_{|n|\leq k} a_n e^{i\theta n}$ is non-negative if and only if $Q(\theta)=|P(e^{i\theta})|^2$ for a polynomial $P$ of degree at most $k$. 
\begin{lemma}\label{lem:structure} 
	Fix $0<p<2$ and let $L_k$ be a bounded linear functional on $H^p$ such that $L_k(f)$ depends only on the first $k+1$ coefficients of $f(z) = \sum_{n\geq0} a_n z^n$. Any extremal for $L_k$ is given by a sequence $(\alpha_j)_{j=1}^k$ with $|\alpha_j|\leq1$ and a constant $A$ such that 
	\begin{equation} \label{eq:structure} 
		f(z) = A \prod_{j=1}^{l} \frac{z+\alpha_j}{1+\overline{\alpha_j}z} \prod_{j=1}^k (1+\overline{\alpha_j} z)^{2/p}, 
	\end{equation}
	where $0 \leq l \leq k$ and $|\alpha_j|<1$ for $1 \leq j \leq l$. In particular, if $f$ is normalised by $\|f\|_{H^p}$ = 1 and $f=gh^{2/p-1}$ as in \eqref{eq:innerouter}, we have that $h$ and $g$ are polynomials that can be written as 
	\begin{equation} \label{eq:hg}
		h(z) = A_1 \prod_{j=1}^k (1+\overline{\alpha_j} z)\qquad\text{and}\qquad g(z)=A_2 \prod_{j=1}^l (z+\alpha_j)\prod_{j=l+1}^k (1+\overline{\alpha_j} z)
	\end{equation}
	with suitable constants $A_1,A_2$. 
\end{lemma}
\begin{proof}
	We begin by writing $f = g h^{2/p-1}$ as in \eqref{eq:innerouter}. We use Lemma~\ref{lem:variational} with $\varphi(z)=z^n$ to obtain
	\[L_k(z^n f) = L(f)\langle z^n, |h|^2 \rangle.\]
	Since $L_k(z^n f)=0$ for $n>k$, we conclude that $|h|^2$ is a trigonometric polynomial of degree at most $k$. The non-negativity of $|h|^2$ and the Fej\'er--Riesz theorem implies that $|h(e^{i\theta})|^2=|P(e^{i\theta})|^2$ for some polynomial $P$ of degree at most $k$. It is clear that $P = B \widetilde{P}$, where $B$ is a finite Blaschke product and $\widetilde{P}$ is an outer polynomial of degree at most $k$. Since an outer function is determined up to a unimodular constant by its modulus on $\mathbb{T}$, we therefore find that $h = e^{i \vartheta}\widetilde{P}$, which means that
	\[h(z) = A_1 \prod_{j=1}^k (1+\overline{\alpha_j} z),\]
	for $|\alpha_j|\leq1$. Our next goal is to establish that $g$ is also a polynomial of degree at most $k$. Suppose that $h$ is fixed as above and note that $h^{2/p-1}$ is in $H^\infty$ since $2/p-1>0$. The fact that $f$ is extremal for $L_k$ and H\"older's inequality implies that $g$ is an $H^2$ function of unit norm attaining the maximum of
	\begin{equation} \label{eq:gfunc}
		g\mapsto \mre L_k(f) = \mre L_k(g h^{2/p-1}).
	\end{equation}
	It is clear that \eqref{eq:gfunc} defines a bounded linear functional on $H^2$ which depends only on the first $k+1$ coefficients of $g$. The Cauchy--Schwarz inequality then implies that $g$ is a polynomial of degree at most $k$. By \eqref{eq:innerouter}, we recall that $g = I h$ for a inner function $I$ and a polynomial $h$. Clearly this is only possible if the inner function $I$ is a finite Blaschke product of degree $0 \leq l \leq k$. Hence
	\[g(z) = A_2 \prod_{j=1}^l \frac{z+\beta_j}{1+\overline{\beta_j}z} \prod_{j=1}^k (1+\overline{\alpha_j} z),\]
	for $|\beta_j|<1$. Since $g$ is a polynomial, we must have $\beta_j = \alpha_j$ for $1 \leq j \leq l$. 
\end{proof}

Let us now return to the bounded linear functional defined by $L_k(f) = a_k$ for $f(z)=\sum_{n\geq0} a_n z^n$ in $H^p$. In the case $1<p<\infty$, the strict convexity of $H^p$ yields easily that the extremal for $C(k,p)=1$ is $f(z) = z^k$. Hence $h(z)=1$ and $g(z)=z^k$ in \eqref{eq:hg}. In the case $p=1$ it is known (see e.g.~\cite[p.~143]{Duren}) that every function of the form \eqref{eq:structure} is an extremal for $C(k,1)=1$. 

For $0<p<1$, we can factor the extremal as
\[f = g h^{2/p-1},\]
where $g$ and $h$ are polynomials related by \eqref{eq:hg}. Our plan is to consider each of the cases $l=0,\ldots,k$ in Lemma~\ref{lem:structure} through the Cauchy product \eqref{eq:cauchyproduct}. Since we may assume that $\|f\|_{H^p}=\|g\|_{H^2}=\|h\|_{H^2}=1$ for any extremal $f$, there must be a constant $\lambda$ such that the equation 
\begin{equation}\label{eq:flipeq} 
	\lambda z^k g(z^{-1}) = \overline{h^{2/p-1}(\overline{z})}+O(z^{k+1}). 
\end{equation}
holds. Namely, otherwise we could modify $g$ to obtain equality in Cauchy--Schwarz in \eqref{eq:cauchyproduct} while keeping $\|g\|_{H^2}=1$ and a fortiori $\|f\|_{H^p}\leq 1$, by H\"older's inequality. By the same argument, it follows that any such (not necessarily normalized) solution of the equation \eqref{eq:flipeq} satisfies
\begin{equation} \label{eq:cprod}
	L_k(f) = \sum_{j=0}^k b_j c_{k-j} = \lambda \sum_{j=0}^k |b_j|^2 = \lambda \|g\|_{H^2}^2.
\end{equation}
In practice this approach will yield a non-linear system of $k+1$ equations in the $k+1$ unknowns which needs to be solved in order to identify the candidate extremal function. We complete the program by comparing the solutions for $l=0,\ldots,k$. 

Using Lemma~\ref{lem:structure} and \eqref{eq:flipeq} in this way, it is possible to give a (computationally) simpler proof of \eqref{eq:C1p} compared to the one given in \cite[Thm.~4.1]{BBSS}.

\section{Proof of Theorem~\ref{thm:C2p}} \label{sec:C2p} 
For $0<p<1$ define $q=2/p-1>1$. For the functional $L_2(f)=a_2$ we get from Lemma~\ref{lem:structure} that the extremal functions are of the form 
\begin{multline*}
	f(z) = A \prod_{j=1}^l \frac{z+\alpha_j}{1+\overline{\alpha_j}z} \prod_{j=1}^2 (1+\overline{\alpha_j}z)^{2/p} \\
	=A \prod_{j=1}^l (z+\alpha_j) \prod_{j=l+1}^2 (1+\overline{\alpha_j}z) \prod_{j=1}^2 (1+\overline{\alpha_j}z)^q = A g(z) (h(z))^q, 
\end{multline*}
where $|\alpha_j|\leq1$ with strict inequality for $1 \leq j \leq l$. We get three equations from $l=0,1,2$. Recall that $\|g\|_{H^2}=\|h\|_{H^2}$, so the normalizing constant is $A = \|h\|_{H^2}^{-2/p}$. We begin by computing
\[\overline{(h(\overline{z}))^q} = 1 + q\beta z + \left(\binom{q}{2}\beta^2 + q \alpha\right)z^2 + O(z^3),\]
where $\alpha=\alpha_1\alpha_2$ and $\beta = \alpha_1+\alpha_2$. Hence the equation \eqref{eq:flipeq} becomes 
\begin{equation}\label{eq:C2peq} 
	\lambda z^2 g(z^{-1}) = 1 + q\beta z + \left(\binom{q}{2}\beta^2 + q \alpha\right)z^2. 
\end{equation}
Note that if $f$ is a normalized solution of the equation \eqref{eq:C2peq}, then by \eqref{eq:cprod} we get
\begin{equation}\label{eq:C2pLf} 
	a_2= L_2(f) = A |\lambda| \|g\|_{H^2}^2 = |\lambda| \|h\|_{H^2}^{2(1-1/p)} = |\lambda| \left(1+|\beta|^2+|\alpha|^2\right)^{1-1/p}. 
\end{equation}

\subsection*{The case $l=2$} Here we have
\[g(z) = (z+\alpha_1)(z+\alpha_2) = z^2 + \beta z+\alpha,\]
so the equation \eqref{eq:C2peq} takes the form:
\[\lambda =1 \qquad \lambda\beta = q\beta \qquad \lambda \alpha = \binom{q}{2}\beta^2 + q\alpha\]
Recalling that $q>1$ we conclude that $\alpha=\beta=0$. Hence $\alpha_1=\alpha_2=0$ and the normalized candidate extremal function function is $f(z)=z^2$ which has $a_2=1$.

\subsection*{The case $l=1$} Here we have
\[g(z) = (z+\alpha_1)(1+\overline{\alpha_2}z) = \overline{\alpha_2}z^2 + (1+\alpha_1\overline{\alpha_2})z+\alpha_1.\]
By a rotation, we assume that $\alpha_2\geq0$ and hence the equation \eqref{eq:C2peq} takes the form: 
\begin{align}
	\lambda \alpha_2 &= 1 \label{eq:C2pl11} \\
	\lambda(1+\alpha_1\alpha_2) &= q(\alpha_1+\alpha_2) \label{eq:C2pl12} \\
	\lambda \alpha_1 &= {\textstyle\binom{q}{2}}(\alpha_1+\alpha_2)^2 + q\alpha_1\alpha_2 
\label{eq:C2pl13} \end{align}
From \eqref{eq:C2pl11} we get that $\alpha_2=\lambda^{-1}>0$. Inserting this into \eqref{eq:C2pl12} yields that 
\begin{equation} \label{eq:C2pl12A} 
	\frac{1}{\alpha_2}+\alpha_1 = q(\alpha_1+\alpha_2). 
\end{equation}
Since $q>1$ we now see that $\alpha_1$ is real. We then multiply \eqref{eq:C2pl12A} with $\alpha_1$ and rearrange to obtain $\lambda \alpha_1 - q \alpha_1 \alpha_2 = (q-1)\alpha_1^2$, which when inserted into \eqref{eq:C2pl13} yields 
\[\frac{2}{q} \alpha_1^2 = (\alpha_1+\alpha_2)^2.\]
Taking the square root of this we find that
\[\alpha_2 = \alpha_1\left(-1\pm\sqrt{\frac{2}{q}}\right) \qquad\text{and}\qquad \frac{1}{\alpha_2} = \alpha_1\left(-1\pm \sqrt{2q}\right),\]
where the second equality was obtained by inserting the first into \eqref{eq:C2pl12A}. Note that for $1< q \leq 2$ we see from the second equation that we have to choose the negative sign to ensure that $|\alpha_1 \alpha_2|<1$. In the range $2<q<\infty$ we also have to choose the negative sign to ensure that the sign requirement $\alpha_1<0$ from first equation also holds in the second. In particular, we get that $\alpha_1<0$ in general. Evidently, 
\begin{equation}\label{eq:C2pl1alphas} 
	\alpha_1^2 = \frac{1}{\big(1+\sqrt{2/q}\big)(1+\sqrt{2q})} \qquad \text{and}\qquad \alpha_2^2 = \frac{1+\sqrt{2/q}}{1+\sqrt{2q}}. 
\end{equation}
Recalling that $\lambda = \alpha_2^{-1}$, we get from \eqref{eq:C2pLf} that the normalized candidate extremal function $f$ satisfies 
\begin{equation}\label{eq:C2pl1Lf} 
	a_2 = L_2(f) = \frac{1}{\alpha_2} \left(1 + (\alpha_1+\alpha_2)^2 + (\alpha_1\alpha_2)^2\right)^{1-1/p}. 
\end{equation}

\subsection*{The case $l=0$} Here we have
\[g(z) = (1+\overline{\alpha_1}z)(1+\overline{\alpha_2}z) = \overline{\alpha} \,z^2 + \overline{\beta}\,z + 1.\]
If $\beta=0$ we get the extremal \eqref{eq:C1pextremal} for $C(1,p)$ with the argument squared. Assume therefore that $\beta\neq0$. There are two rotations $e^{i\theta}$ and $e^{i(\theta+\pi)}$ such that $\alpha\geq0$. The equation \eqref{eq:C2peq} takes the form: 
\begin{align}
	\lambda \alpha &= 1 \label{eq:C2pl01} \\
	\lambda \overline{\beta} &= q \beta \label{eq:C2pl02} \\
	\lambda &= {\textstyle \binom{q}{2}}\beta^2 + q\alpha 
\label{eq:C2pl03} \end{align}
From \eqref{eq:C2pl01} we get that $\lambda = \alpha^{-1}>0$. Since $\alpha,\lambda,q>0$ we get from \eqref{eq:C2pl03} that $\beta^2$ is real, and hence $\beta$ is real or imaginary. By \eqref{eq:C2pl02} we see that $\beta$ cannot be imaginary, since $\lambda,q>0$. We conclude that $\beta$ is real. Choosing the appropriate rotation above we get that $\beta>0$. Combining \eqref{eq:C2pl01} and \eqref{eq:C2pl02} yields that $\alpha = \lambda^{-1} = q^{-1}$. Inserting this into \eqref{eq:C2pl03} we find that
\[q = \binom{q}{2} \beta^2 +1 \qquad \implies \qquad \beta = \sqrt{\frac{2}{q}}.\]
We get from \eqref{eq:C2pLf} that the normalized candidate extremal function satisfies 
\begin{equation}\label{eq:C2pl2Lf} 
	a_2 = L_2(f) = q \left(1+\frac{2}{q}+\frac{1}{q^2}\right)^{1-1/p}. 
\end{equation}

\subsection*{Final part in the proof of Theorem~\ref{thm:C2p}} We need to compare the normalized candidate extremal functions from the equations $l=0,1,2$. Clearly $a_2=1$ from $l=2$ can be discarded at once. Comparing \eqref{eq:C2pl1Lf} and \eqref{eq:C2pl2Lf}, we claim that
\[\frac{1}{\alpha_2} \left(1 + (\alpha_1+\alpha_2)^2 + (\alpha_1\alpha_2)^2\right)^{1-1/p} \leq q \left(1+\frac{2}{q}+\frac{1}{q^2}\right)^{1-1/p},\]
where $\alpha_1$ and $\alpha_2$ are given by \eqref{eq:C2pl1alphas}. We recall that $1-1/p<0$, so a stronger statement is
\[1 \leq \alpha_2 q \left(1+\frac{2}{q}+\frac{1}{q^2}\right)^{1-1/p} = \sqrt{\frac{1+\sqrt{2/q}}{1+\sqrt{2q}}}q\left(1+\frac{1}{q}\right)^{1-q} = \Phi(q),\]
where we used that $2/p-1=q$. Note that $\Phi(1)=1$. We compute
\[\frac{d}{dq}\log{\Phi(q)} = -\frac{1}{2\sqrt{2q}}\left(\frac{1}{q+\sqrt{2q}}+\frac{1}{1+\sqrt{2q}}\right) + \frac{2}{1+q} - \log\left(1+\frac{1}{q}\right).\]
For $q\geq1$ it holds that $q+\sqrt{2q} \geq 1+\sqrt{2q}$, so
\[-\frac{1}{2\sqrt{2q}}\left(\frac{1}{q+\sqrt{2q}}+\frac{1}{1+\sqrt{2q}}\right) \geq -\frac{1}{\sqrt{2q}+2q} \geq -\frac{1}{\sqrt{2}+2q} \geq -\frac{2-\sqrt{2}}{1+q}.\]
The final inequality is easily checked directly. Consequently
\[\frac{d}{dq}\log{\Phi(q)} \geq \frac{\sqrt{2}}{1+q}-\log\left(1+\frac{1}{q}\right)=\Psi(q).\]
We get that $\Phi$ is increasing on $1<q<\infty$ by proving that $\Psi(q)>0$ in the same range, which can be deduced by checking the non-negativity of $\Psi$ in the endpoints and at the critical point $q=1+\sqrt{2}$. Hence we conclude that the case $l=0$ provides the extremal function and that
\[C(2,p) = q \left(1+\frac{2}{q}+\frac{1}{q^2}\right)^{1-1/p} = \frac{2}{p}\left(1-\frac{p}{2}\right)^{\frac{2}{p}-1}.\]
In the case $l=0$ we have that $g(z)=h(z)=1+\beta z + \alpha z^2$, so a computation yields the stated extremal function. \qed

\section{Proof of Theorem~\ref{thm:C323}} \label{sec:C323} 
By Lemma~\ref{lem:structure}, we get that the candidate extremal functions for the functional $L_3(f)=a_3$ acting on $H^p$ with $p=2/3$ are of the form 
\begin{multline*}
	f(z) = A \prod_{j=1}^l \frac{z+\alpha_j}{1+\overline{\alpha_j}z} \prod_{j=1}^3 (1+\overline{\alpha_j}z)^3 \\
	= A \prod_{j=1}^l (z+\alpha_j) \prod_{j=l+1}^3 (1+\overline{\alpha_j}z) \prod_{j=1}^3 (1+\overline{\alpha_j})^2 = A g(z) (h(z))^2, 
\end{multline*}
where $|\alpha_j|\leq1$ with strict inequality for $1 \leq j \leq l$. There are four equations, from $l=0,1,2,3$. Recall that $\|g\|_{H^2}=\|h\|_{H^2}$ and that the normalizing constant is $A = \|h\|_{H^2}^{-3}$. We begin by computing
\[\overline{(h(\overline{z}))^2} = 1 + 2 \beta z + \left(\beta^2 + 2\gamma\right)z^2 + 2\left(\beta\gamma+\alpha\right)z^3 + O(z^4)\]
where $\alpha=\alpha_1\alpha_2\alpha_3$, $\beta = \alpha_1+\alpha_2+\alpha_3$ and $\gamma = \alpha_1\alpha_2+\alpha_1\alpha_3+\alpha_2\alpha_3$. Hence the equation \eqref{eq:flipeq} becomes 
\begin{equation}\label{eq:C323eq} 
	\lambda z^3 g(z^{-1}) = 1 + 2 \beta z + \left(\beta^2 + 2\gamma\right)z^2 + 2\left(\beta\gamma+\alpha\right)z^3. 
\end{equation}
Note that if $f$ is a normalized solution to the equation \eqref{eq:C323eq}, then by \eqref{eq:cprod} we get
\begin{equation}\label{eq:C323Lf} 
	a_3 = L_3(f) = A |\lambda| \|g\|_{H^2}^2 = |\lambda| \|h\|_{H^2}^{-1} = |\lambda| \left(1+|\beta|^2+|\gamma|^2+|\alpha|^2\right)^{-1/2}. 
\end{equation}

\subsection*{The case $l=3$} Here we get
\[g(z) = (z+\alpha_1)(z+\alpha_2)(z+\alpha_3) = z^3 + \beta z^2 + \gamma z + \alpha,\]
which means that the equation \eqref{eq:C323eq} takes the form:
\[\lambda = 1 \qquad \lambda \beta = 2 \beta \qquad \lambda \gamma = \beta^2+2\gamma \qquad \lambda \alpha = 2\left(\beta\gamma+\alpha\right)\]
The only solution is $\alpha=\beta=\gamma=0$, which implies $\alpha_1=\alpha_2=\alpha_3=0$. The normalized candidate extremal function is $f(z)=z^3$, which has $a_3=1$.

\subsection*{The case $l=2$} Here we get 
\begin{multline*}
	g(z) = (z+\alpha_1)(z+\alpha_2)(1+\overline{\alpha_3}z) \\
	=\overline{\alpha_3}z^3 + \left((\alpha_1+\alpha_2)\overline{\alpha_3}+1\right)z^2 + \left(\alpha_1\alpha_2\overline{\alpha_3}+\alpha_1+\alpha_2\right)z+\alpha_1\alpha_2. 
\end{multline*}
Set $\xi=\alpha_1\alpha_2$, $\eta = \alpha_1+\alpha_2$ and $\alpha_3=\varrho$. By a rotation, we may assume that $\varrho\geq0$. The equation \eqref{eq:C323eq} takes the form: 
\begin{align}
	\lambda \varrho &= 1 \label{eq:C323l21} \\
	\lambda(\eta\varrho+1) &= 2(\eta+\varrho) \label{eq:C323l22} \\
	\lambda(\xi\varrho+\eta) &= (\eta+\varrho)^2 + 2(\xi+\eta\varrho) \label{eq:C323l23} \\
	\lambda \xi &= 2\big((\eta+\varrho)(\xi+\eta\varrho)+\xi\varrho\big) 
\label{eq:C323l24} \end{align}
From \eqref{eq:C323l21} we get that $\varrho>0$. Inserting \eqref{eq:C323l21} into \eqref{eq:C323l22} and solving for $\eta$ yields that 
\begin{equation}\label{eq:betatilde} 
	\eta = \frac{1}{\varrho}-2\varrho. 
\end{equation}
Inserting \eqref{eq:C323l21} into \eqref{eq:C323l23} and solving for $\xi$ yields that 
\begin{equation}\label{eq:alphatilde} 
	\xi = \frac{\eta}{\varrho} - 2 \eta\varrho-(\eta+\varrho)^2 = \frac{1}{\varrho^2}-2 -(1-2\varrho^2)-\left(\frac{1}{\varrho}-\varrho\right)^2 = 3\varrho^2-2, 
\end{equation}
where we in the penultimate equality used \eqref{eq:betatilde}. Inserting \eqref{eq:C323l21}, \eqref{eq:betatilde} and \eqref{eq:alphatilde} into \eqref{eq:C323l24} now yields
\[3\varrho-\frac{2}{\varrho} = 2 \left(\left(\frac{1}{\varrho}-\varrho\right)(\varrho^2-1)+(3\varrho^2-2)\varrho\right) = 4\varrho^3-\frac{2}{\varrho}.\]
Since $\varrho>0$ we get that $\varrho = \sqrt{3}/2$, which by \eqref{eq:betatilde} and \eqref{eq:alphatilde} implies that $\eta=-\sqrt{3}/3$ and $\xi=1/4$, respectively. Recalling that $\lambda = \varrho^{-1}$, $\alpha = \xi \varrho$, $\beta = \eta+\varrho$ and $\gamma=\xi+\eta\varrho$, we get from \eqref{eq:C323Lf} that the normalized candidate extremal function $f$ satisfies 
\begin{equation}\label{eq:C323l2Lf} 
	a_3 = L_3(f) = \frac{1}{\varrho}\left(1 + (\eta+\varrho)^2 + (\xi+\eta\varrho)^2 + (\xi \varrho)^2\right)^{-1/2} = \frac{16}{\sqrt{229}} = 1.0573\ldots 
\end{equation}

\subsection*{The case $l=1$} Here we get 
\begin{multline*}
	g(z) = (z+\alpha_1) (1+\overline{\alpha_2}z)(1+\overline{\alpha_3}z) \\
	= z^3 \overline{\alpha_2\alpha_3} + z^2(\overline{\alpha_2+\alpha_3}+\alpha_1\overline{\alpha_2\alpha_3}) + z(1+\alpha_1(\overline{\alpha_2+\alpha_3})) + \alpha_1. 
\end{multline*}
Set $\varrho = \alpha_1$, $\eta=\alpha_2+\alpha_3$ and $\xi=\alpha_2\alpha_3$. There are four rotations $e^{i\theta}$, $e^{i(\theta\pm\pi/2)}$ and $e^{i(\theta+\pi)}$ such that $\xi$ is real. The equation \eqref{eq:C323eq} then takes the form: 
\begin{align}
	\lambda \xi &= 1 \label{eq:C323l11} \\
	\lambda (\overline{\eta}+\varrho\xi) &= 2(\varrho+\eta) \label{eq:C323l12} \\
	\lambda (1+\varrho\overline{\eta}) &= (\varrho+\eta)^2 + 2(\varrho\eta + \xi) \label{eq:C323l13} \\
	\lambda \varrho &= 2\big((\varrho+\eta)(\varrho\eta+\xi)+\varrho\xi\big) 
\label{eq:C323l14} \end{align}
From \eqref{eq:C323l11} we get that $\xi\neq0$ and $\lambda = \xi^{-1}$. Inserting this into \eqref{eq:C323l12}, we obtain 
\begin{equation}\label{eq:varrhoeq} 
	\varrho = \frac{\overline{\eta}}{\xi}-2\eta. 
\end{equation}
Inserting \eqref{eq:C323l11} and \eqref{eq:varrhoeq} into \eqref{eq:C323l13}, we obtain 
\begin{multline*}
	\frac{1}{\xi} + \frac{\overline{\eta}^2}{\xi^2}-\frac{2|\eta|^2}{\xi} = \left(\frac{\overline{\eta}}{\xi}-\eta\right)^2 + 2\left(\frac{|\eta|^2}{\xi}-2\eta^2 + \xi \right) = \frac{\overline{\eta}^2}{\xi^2}-3\eta^2 + 2\xi \\
	\qquad \Longleftrightarrow \qquad \frac{1}{\xi}-\frac{2|\eta|^2}{\xi} = 2\xi -3\eta^2. 
\end{multline*}
Hence we find that $\eta^2$ is real. By choosing the appropriate rotation above, we may assume that $\eta\geq0$, in which case it holds that 
\begin{equation}\label{eq:etaeq} 
	\eta = \sqrt{\frac{1-2\xi^2}{2-3\xi}}.
\end{equation}
We then insert \eqref{eq:C323l11} and \eqref{eq:varrhoeq} into \eqref{eq:C323l14}, keeping in mind that $\eta\geq0$, to obtain 
\begin{equation}\label{eq:prediveta} 
	\frac{\eta}{\xi}\left(\frac{1}{\xi}-2\right)=2 \left(\eta \left(\frac{1}{\xi}-1\right)\left(\eta^2 \left(\frac{1}{\xi}-2\right)+\xi\right)+\eta(1-2\xi)\right). 
\end{equation}
The equation \eqref{eq:prediveta} with $\eta$ as in \eqref{eq:etaeq} has five real solutions. Before we compute them, let us recall that that $\beta = \varrho + \eta$, $\gamma = \varrho\eta+\xi$ and $\alpha = \varrho \xi$, so we get from \eqref{eq:C323l2Lf} that in each case the normalized candidate extremal function $f$ satisfies 
\begin{equation}\label{eq:a3C323l1} 
	a_3 = L(f) = \frac{1}{|\xi|} \left(1+(\varrho+\eta)^2 + (\varrho\eta+\xi)^2 + (\varrho \xi)^2\right)^{-1/2}. 
\end{equation}
The first two solutions of \eqref{eq:prediveta} arise from the case $\eta=0$, which occurs when $\varrho=0$ and $\xi^2 = 1/2$. Here we easily find from \eqref{eq:a3C323l1} that 
\begin{equation}\label{eq:C323l1Lf1} 
	a_3 = \frac{2}{\sqrt{3}}=1.1547\ldots. 
\end{equation}
If $\eta\neq0$, we may multiply \eqref{eq:prediveta} by $(2-3\xi)\xi/\eta$, then insert the value for $\eta^2$ and simplify to obtain
\[10 \xi^3-12\xi^2 + 2 \xi + 1 = 0.\]
This equation has the following solutions: 
\begin{align*}
	\xi_1 &= {\textstyle\frac{2}{5}}\Big(1-\sqrt{\textstyle\frac{7}{3}}\cos{\vartheta}\Big) = -0.2049\ldots \\
	\xi_2 &= {\textstyle\frac{1}{5}}\Big(2+\sqrt{\textstyle\frac{7}{3}}\big(\cos{\vartheta}-\sqrt{3}\sin{\vartheta}\big)\Big) = 0.6281\ldots \qquad\qquad \text{for } \vartheta = {\textstyle\frac{1}{3}}\arctan\Big({\textstyle\frac{5\sqrt{111}}{117}}\Big)\\
	\xi_3 &= {\textstyle\frac{1}{5}}\Big(2+\sqrt{\textstyle\frac{7}{3}}\big(\cos{\vartheta}+\sqrt{3}\sin{\vartheta}\big)\Big) = 0.7768\ldots 
\end{align*}
Inserting these and the corresponding $\varrho$ and $\eta$ into \eqref{eq:a3C323l1} yields, respectively, 
\begin{equation}\label{eq:C323l1Lf2} 
	a_3 = 1.0739\ldots \qquad a_3 = 1.1958\ldots \qquad a_3 = 1.1067\ldots 
\end{equation}

\subsection*{The case $l=0$} Here we get
\[g(z) = (1+\overline{\alpha_1}z)(1+\overline{\alpha_2}z)(1+\overline{\alpha_3}z) = \overline{\alpha}\,z^3 + \overline{\gamma}\,z^2 + \overline{\beta}\,z +1.\]
There are three rotations, $e^{i\theta}$, $e^{i(\theta+\pi/3)}$ and $e^{i(\theta+2\pi/3)}$ such that $\alpha=\alpha_1\alpha_2\alpha_3\geq0$. The equation \eqref{eq:C323eq} takes the form:
\[\lambda\alpha = 1 \qquad \lambda \overline{\gamma} = 2\beta \qquad \lambda \overline{\beta} = \beta^2+2\gamma \qquad \lambda = 2\left(\beta\gamma+\alpha\right)\]
The first equation shows that $\alpha>0$. We insert it into the others and obtain: 
\begin{align}
	\overline{\gamma} &= 2\alpha \beta \label{eq:C323l01} \\
	\overline{\beta} &= \alpha\beta^2 + 2\alpha\gamma \label{eq:C323l02} \\
	1 &= 2 (\alpha \beta \gamma + \alpha^2) 
\label{eq:C323l03} \end{align}
Our goal is to show that $\beta$ (and hence $\gamma$) is real. We begin with \eqref{eq:C323l02}. Inserting the conjugate of \eqref{eq:C323l01}, multiplying with $\beta$ and applying \eqref{eq:C323l03} yields
\[\alpha \beta^2 = \frac{\gamma}{2\alpha} - 2 \alpha \gamma = \gamma \left(\frac{1}{2\alpha}-2\alpha\right) \qquad\implies\qquad \alpha \beta^3 = \frac{1-2\alpha^2}{2\alpha}\left(\frac{1}{2\alpha}-2\alpha\right).\]
Hence $\beta^3$ is real, so we may choose a rotation above to ensure that $\beta$ is real. Note now that $\beta =0$ if and only if $\gamma=0$, which leads to the extremal \eqref{eq:C1pextremal} for $C(1,2/3)$ with the argument cubed. Hence we assume $\beta\neq0$. Since know that $\beta$ and $\gamma$ are real and non-zero, we insert \eqref{eq:C323l01} into \eqref{eq:C323l02} to obtain that
\[\beta = \alpha \beta^2 + 4\alpha^2 \beta \qquad \implies \qquad\beta = \frac{1-4\alpha^2}{\alpha} \qquad \implies \qquad \gamma =2-8\alpha^2,\]
where we used \eqref{eq:C323l01} again for the second implication. Inserting the values for $\beta$ and $\gamma$ into \eqref{eq:C323l03} yields the equation $1 = 2(2(1-4\alpha^2)^2+\alpha^2)$. Since $\alpha>0$ there are only two solutions:
\[\alpha = \frac{\sqrt{15\pm \sqrt{33}}}{8} \qquad \beta = \mp \frac{\sqrt{3\mp\frac{1}{3}\sqrt{33}}}{2} \qquad \gamma =\frac{1\mp \sqrt{33}}{8}.\]
Recalling that $\lambda = \alpha^{-1}$, we get from \eqref{eq:C323Lf} that the normalized candidate extremal function $f$ satisfies 
\begin{equation}\label{eq:C323l0Lf} 
	a_3 = L_3(f) = \frac{1}{\alpha}\left(1+\beta^2+\gamma^2+\alpha^2\right)^{-1/2} = \sqrt{\frac{2\left(1103\mp33\sqrt{33}\,\right)}{1153}}. 
\end{equation}
To maximize this, we choose the negative sign in the expression for $\alpha$, which yields that $\beta,\gamma>0$ and the value $a_3 = 1.4973\ldots$ in \eqref{eq:C323l0Lf}.

\subsection*{Final part in the proof of Theorem~\ref{thm:C323}} We need to compare the candidate extremal functions from the equations $l=0,1,2,3$. Clearly $a_3=1$ from $l=3$ can be discarded at once. Comparing \eqref{eq:C323l2Lf}, \eqref{eq:C323l1Lf1}, \eqref{eq:C323l1Lf2} and \eqref{eq:C323l0Lf} we find that the latter is the largest. Hence the case $l=0$ provides the extremal function so that
\[C(3,2/3) = \sqrt{\frac{2\left(1103+33\sqrt{33}\,\right)}{1153}}.\]
In the case $l=0$ we have $g(z)=h(z)=1+\beta z + \gamma z^2 + \alpha z^3$, so a computation yields the stated extremal function. \qed

\section{Concluding remarks} \label{sec:remarks}

\subsection{} Our first observation is that neither the extremal for $C(1,p)$ from \eqref{eq:C1pextremal} nor the extremals for $C(2,p)$ and $C(3,2/3)$ from Theorem~\ref{thm:C2p} and Theorem~\ref{thm:C323}, respectively, vanish in $\mathbb{D}$. This is of course a consequence of the fact that the extremals in each case stem from the case $l=0$ in Lemma~\ref{lem:structure}.
\begin{conjecture}\label{conj:nonvanish} 
	For $0<p<1$ any extremal $f$ for $C(k,p)$ does not vanish in $\mathbb{D}$. 
\end{conjecture}

If we a priori knew that Conjecture~\ref{conj:nonvanish} held, it would significantly decrease the effort needed to prove Theorem~\ref{thm:C2p} and Theorem~\ref{thm:C323}, since it would be sufficient to consider only the case $l=0$. Apart from the above-mentioned examples we have little concrete evidence for the conjecture. However, the following weaker statement could be a starting point. 
\begin{conjecture}\label{conj:strict} 
	For $0<p<1$ the sequence $C(k,p)$ is strictly increasing. 
\end{conjecture}

Conjecture~\ref{conj:strict} is equivalent to the following statement: For $0<p<1$ any extremal for $C(k,p)$ does not vanish at the origin. Indeed, if $C(k,p)=C(k+1,p)$ for some $k\geq1$ then we can multiply an extremal for $C(k,p)$ with $z$ to obtain an extremal for $C(k+1,p)$ vanishing at the origin. Conversely, if an extremal for $C(k+1,p)$ vanishes at the origin, then we find that $C(k,p)=C(k+1,p)$ by dividing the extremal by $z$. Note that this is precisely how the extremals $f(z)=z^k$ can be obtained in the range $1\leq p<\infty$, where it holds that $C(k,p)=1$ for every $k$.

\subsection{} Let $N_p$ denote the subset of $H^p$ consisting of the elements $f$ which do not vanish in $\mathbb{D}$. Suffridge \cite{Suffridge90} investigated the extremal problem
\[\widetilde{C}(k,p) = \sup_{f\in N_p} \left\{\mre{\frac{f^{(k)}(0)}{k!}}\,:\, \|f\|_{H^p}=1\right\}.\]
Clearly it holds that $\widetilde{C}(k,p) \leq C(k,p)$. By Lemma~\ref{lem:structure} (see also \cite[p.~143]{Duren}) this is an equality when $p=1$. For $1<p<\infty$ this inequality is strict, by the strict convexity of $H^p$ and the fact that $f(z)=z^k$ are not in $N^p$.

Note that Conjecture~\ref{conj:nonvanish} is equivalent to the claim $\widetilde{C}(k,p)=C(k,p)$ for $0<p<1$ and $k\geq1$. In particular, we observe that \cite[Thm.~4.1]{BBSS} and Theorem~\ref{thm:C2p} extend the statements for $0<p<1$ in \cite[Thm.~2]{Suffridge90} and \cite[Thm.~7]{Suffridge90}, respectively.

The approach employed in \cite{Suffridge90} to study $\widetilde{C}(k,p)$ is related to the approach of the present paper to study $C(k,p)$. The difference is that the version of Lemma~\ref{lem:structure} for $N_p$ does not contain a Blaschke product, but instead contains a singular inner function. It is conjectured on \cite[p.~187]{Suffridge90} that this singular inner function is trivial when $0<p<1$. This conjecture is evidently a consequence of Conjecture~\ref{conj:nonvanish} in view of Lemma~\ref{lem:structure}.

\subsection{} Fix $0\leq r \leq 1$ and let $H^p_r$ denote the subset of $H^p$ consisting of the elements $f$ for which $|f(0)|=r$. For $k\geq1$, consider the extremal problem 
\[C_r(k,p) = \sup_{f \in H^p_r} \left\{\mre{\frac{f^{(k)}(0)}{k!}}\,:\, \|f\|_{H^p}=1 \right\}.\]
This extremal problem was solved by Beneteau and Korenblum \cite{BK04} in the range $1\leq p < \infty$ as follows. They first demonstrate that $C_r(k,p)=C_r(1,p)$ holds for every $k\geq1$ using F.~Wiener's trick, which relies on the triangle inequality. Following this, they solve the extremal problem directly in the case $k=1$ using the factorization $f = BF$ similarly to how we used the factorization $f = gh^{2/p-1}$ above. Inspecting the solution, it is easy to verify that the function $r \mapsto C_r(k,p)$ is decreasing from $C_0(k,p)=1$ to $C_1(k,p)=0$. 

We make a couple of comments on this extremal problem in the range $0<p<1$. Since the triangle inequality here takes the form
\[\|f+g\|_{H^p}^p \leq \|f\|_{H^p}^p + \|g\|_{H^p}^p,\]
we find by F.~Wiener's trick that $C_r(k,p) \leq k^{1/p-1}C_r(1,p)$. This estimate should be compared with the Hardy--Littlewood estimate $C(k,p) \leq k^{1/p-1} C(1,p)$ mentioned in the introduction. The situation for $k=1$ is also different, since by \eqref{eq:C1pextremal} and \eqref{eq:C1p} we find that the maxima of the function $r \mapsto C_r(1,p)$ is in the range $0<p<1$ attained at $r = (1-p/2)^{1/p}$.

\subsection{} The dual space of $H^p$ with $0<p<1$, is (non-isometrically) identified in \cite{DRS69} through the embedding
\[\int_{\mathbb{D}} |f(z)| \big({\textstyle\frac{1}{p}}-1\big)\big(1-|z|^2\big)^{\frac{1}{p}-2}\,\frac{dA(z)}{\pi} \leq C_p \|f\|_{H^p},\]
where $dA$ denotes Lebesgue area measure and $C_p\geq1$. The embedding is, of course, also due to Hardy and Littlewood \cite{HL32}. It is conjectured (see e.g.~\cite[Sec.~2]{BOCSZ18}) that $C_p=1$ for every $0<p<1$, but this is known to hold only when $1/p$ is an integer. Assuming that this conjecture holds, we can obtain the estimate
\[C(k,p) \leq \left(2\big({\textstyle\frac{1}{p}}-1\big)\int_0^1 r^{k+1}\big(1-r^2\big)^{\frac{1}{p}-2}\,dr\right)^{-1} = \frac{\Gamma\big(\frac{k}{2}+\frac{1}{p}\big)}{\Gamma\big(\frac{k}{2}+1\big)\Gamma\big(\frac{1}{p}\big)}.\]
For comparison with Theorem~\ref{thm:C2p} and Theorem~\ref{thm:C323}, we record the special cases
\[C(2,p) \leq \frac{1}{p} \qquad \text{and} \qquad C(3,2/3) \leq \frac{16}{3\pi}=1.6976\ldots\]

\section*{Acknowledgements} The authors would like to extend their gratitude to Kristian Seip for several helpful conversations in the early phase of the project and to the anonymous referee for pointing out an inaccuracy in a draft of the paper and for helpful comments that increased the readability of the paper.

\bibliographystyle{amsplain} 
\bibliography{coeff} 

\providecommand{\bysame}{\leavevmode\hbox to3em{\hrulefill}\thinspace}
\providecommand{\MR}{\relax\ifhmode\unskip\space\fi MR }
\providecommand{\MRhref}[2]{%
  \href{http://www.ams.org/mathscinet-getitem?mr=#1}{#2}
}
\providecommand{\href}[2]{#2}
\begin{thebibliography}{10}

\bibitem{BK04}
Catherine Beneteau and Boris Korenblum, \emph{Some coefficient estimates for
  {$H^p$} functions}, Complex analysis and dynamical systems, Contemp. Math.,
  vol. 364, Amer. Math. Soc., Providence, RI, 2004, pp.~5--14.

\bibitem{BBSS}
Andriy Bondarenko, Ole~Fredrik Brevig, Eero Saksman, and Kristian Seip,
  \emph{Linear space properties of {$H^p$} spaces of {Dirichlet} series},
  Trans. Amer. Math. Soc. \textbf{372} (2019), no.~9, 6677--6702.

\bibitem{BOCSZ18}
Ole~Fredrik Brevig, Joaquim Ortega-Cerd\`a, Kristian Seip, and Jing Zhao,
  \emph{Contractive inequalities for {H}ardy spaces}, Funct. Approx. Comment.
  Math. \textbf{59} (2018), no.~1, 41--56.

\bibitem{DRS69}
P.~L. Duren, B.~W. Romberg, and A.~L. Shields, \emph{Linear functionals on
  {$H^{p}$} spaces with {$0<p<1$}}, J. Reine Angew. Math. \textbf{238} (1969),
  32--60.

\bibitem{Duren}
Peter~L. Duren, \emph{Theory of {$H^{p}$} spaces}, Pure and Applied
  Mathematics, Vol. 38, Academic Press, New York-London, 1970.

\bibitem{Fejer16}
Leopold Fej\'{e}r, \emph{\"{U}ber trigonometrische {P}olynome}, J. Reine Angew.
  Math. \textbf{146} (1916), 53--82.

\bibitem{HL32}
G.~H. Hardy and J.~E. Littlewood, \emph{Some properties of fractional
  integrals. {II}}, Math. Z. \textbf{34} (1932), no.~1, 403--439.

\bibitem{Havinson51}
S.~Ya. Havinson, \emph{On some extremal problems of the theory of analytic
  functions}, Moskov. Gos. Univ. U\v{c}enye Zapiski Matematika \textbf{148}
  (1951), no.~4, 133--143.

\bibitem{Kabaila60}
V.~Kabaila, \emph{On some interpolation problems in the class {$H_{p}$} for
  {$p<1$}}, Soviet Math. Dokl. \textbf{1} (1960), 690--692.

\bibitem{Khavinson86}
S.~Ya. Khavinson, \emph{Two papers on extremal problems in complex analysis},
  Amer. Math. Soc. Transl. (2) \textbf{129} (1986), 1--114.

\bibitem{MR50}
A.~J. Macintyre and W.~W. Rogosinski, \emph{Extremum problems in the theory of
  analytic functions}, Acta Math. \textbf{82} (1950), 275--325.

\bibitem{FRiesz19}
Friedrich Riesz, \emph{\"{U}ber {P}otenzreihen mit vorgeschriebenen
  {A}nfangsgliedern}, Acta Math. \textbf{42} (1920), no.~1, 145--171.

\bibitem{Suffridge90}
T.~J. Suffridge, \emph{Extremal problems for nonvanishing {$H^p$} functions},
  Computational methods and function theory ({V}alpara\'\i so, 1989), Lecture
  Notes in Math., vol. 1435, Springer, Berlin, 1990, pp.~177--190.

\end{thebibliography}
\end{document}